\documentclass[12pt, reqno]{amsart}
\usepackage{amsmath}
\usepackage{amssymb,amscd}

\numberwithin{equation}{section}

\newtheorem{theorem}{Theorem}[section]
\newtheorem{corollary}[theorem]{Corollary}
\newtheorem{lemma}[theorem]{Lemma}
\newtheorem{prop}[theorem]{Proposition}

\theoremstyle{definition}

\newtheorem{remark}[theorem]{Remark}
\newtheorem{example}[theorem]{Example}

\begin{document}

\title[  Non-K\"ahler complex manifolds and SKT  structures ]{Group actions, non-K\"ahler complex manifolds and SKT  structures }

\author[M. Poddar]{Mainak Poddar}

\address{Mathematics Group, Middle East Technical University, Northern Cyprus Campus, 99738 Kalkanli, Guzelyurt,  Mersin 10, Turkey}

\email{mainakp@gmail.com}

\author[A. S. Thakur]{Ajay Singh Thakur}

\address{Department of Mathematics and Statistics, Indian Institute of Technology Kanpur, Kanpur 208016, India}

\email{asthakur@iitk.ac.in}



\begin{abstract}  We give a construction of integrable complex structures on the total space of a smooth principal bundle over a complex manifold, with an even dimensional compact Lie group as structure group, under certain conditions. This generalizes the constructions of complex structure on compact Lie groups by Samelson and Wang, and on principal torus bundles by Calabi-Eckmann and others. It also yields  large classes of new examples of non-K\"ahler compact complex manifolds. Moreover, under suitable restrictions on the base manifold, the structure group, and  characteristic classes, the total space of the principal bundle admits SKT metrics. This generalizes recent results of Grantcharov et al. We study the Picard group and the algebraic dimension of the total space in some cases. We also use a slightly generalized version of the construction to obtain (non-K\"ahler) complex structures on tangential frame bundles of complex orbifolds.
\end{abstract}
  
\maketitle

\section{Introduction}\label{intro}
Let $K$ be a compact  Lie group. Let $K_{\mathbb{C}} $ be the complexification of $K$; to wit, $K$ may be regarded as a  Lie subgroup of  $K_{\mathbb{C}} $, and the Lie algebra of   $K_{\mathbb{C}} $ is the complexification of the Lie algebra of $K$.
Then $K_{\mathbb{C}} $ is a reductive complex Lie group and $K$ is a maximal compact subgroup of  $K_{\mathbb{C}} $ (cf. Section 5.3, \cite{akhiezer}). In fact, $K_{\mathbb{C}} $ can be regarded as a reductive complex linear algebraic group. Note that, for a compact Lie group, the notions of complexification and universal complexification coincide. 

We say that a smooth principal $K$-bundle $E_K \rightarrow M$ over a complex manifold $M$  admits a {\it complexification} if the associated principal $K_{\mathbb{C}}$-bundle $E_K \times_K  K_{\mathbb{C}} \rightarrow M$ admits a  holomorphic principal $K_{\mathbb{C}}$-bundle structure. For instance, every smooth flat principal $K$-bundle over a complex manifold admits a complexification. Also, if $M$ is a Stein manifold, then every smooth principal $K$-bundle over $M$ admits a complexification by Grauert's Oka-principle. 
 Note that if $E_K \rightarrow M$ admits a complexification,  then the  transversely holomorphic foliation  induced by the $K$-action on  $E_K$ admits a complexification in the sense of Haefliger-Sundararaman (cf. Remark 4.3, \cite{haefliger}).
 
More generally, let $G$ be a complex linear algebraic group and $K$ be a maximal compact subgroup of $G$.  Then $K$ is a smooth deformation retract of $G$ (cf. Theorem 3.4, Chapter 4, \cite{VGO}). Therefore, any holomorphic principal $G$-bundle $E_G \rightarrow M$ over a complex manifold $M$, when considered as a smooth bundle, admits a smooth reduction of the structure group from $G$ to $K$.  The smooth principal $K$-bundle $E_K \rightarrow M$, corresponding to this reduction, admits a complexification (see Proposition \ref{linear} below). 
 
 Let  $E_K \rightarrow M$  be a smooth principal $K$-bundle  which admits a complexification.  
Suppose $K$ is even dimensional. Then we show that  $E_K \rightarrow M$  has a holomorphic fiber bundle structure (see Theorem 
\ref{main}).  
In addition, if $K$ is a compact torus then  the bundle $E_K \rightarrow M$ has a holomorphic principal $K$-bundle structure. Similarly, if $K$ is odd dimensional,  the bundle $E_K \times S^1 \rightarrow M$ admits a holomorphic fiber bundle structure
with fiber $K \times S^1$.
In particular, the space  $E_K$ admits a normal almost contact structure (nacs).
 Moreover, in most cases, the  space $E_K$ or $E_K \times S^1$  is a non-K\"ahler complex manifold. 
 In this way we obtain new examples of non-K\"ahler complex manifolds; see Corollary \ref{unitary}, Examples \ref{CY}, \ref{Stein}, \ref{toric}, 
 and Remark \ref{oddcase} below.

The above construction of non-K\"ahler complex manifolds generalizes some well-known constructions in the literature.
 For example,  Hopf \cite{hopf} and Calabi-Eckmann 
 \cite{ce} constructed compact non-K\"ahler  complex  manifolds by  obtaining  complex analytic structures on the product of two odd dimensional spheres, $S^{2m-1}\times S^{2n-1}$. Such a product can be viewed as the total space of a smooth principal $S^1 \times S^1$-bundle which admits a complexification,
\begin{equation}\label{tautological}
 (\mathbb C^m \setminus \{0\}) \times (\mathbb C^n \setminus \{0\}) \rightarrow \mathbb C \mathbb P^{m-1} \times \mathbb C \mathbb P^{n-1}.
 \end{equation} Here, the projection map is the component-wise Hopf map. 
The complex structures thus obtained on $S^{2m-1} \times S^{2n-1}$ are those of Hopf and Calabi-Eckmann manifolds. (Moreover, the generalized Calabi-Eckman manifolds considered in \cite{BMT} may be viewed as a special case of our construction.) On the other hand, 
 Samelson \cite{Sam}  and Wang \cite{wang} constructed complex analytic structures on even dimensional compact Lie groups. Our construction  unifies the constructions of Samelson and Wang, and  Calabi-Eckmann, by being a common generalization to both.

 Recently, there have been several other interesting generalizations of Hopf and Calabi-Eckmann manifolds, giving new classes of  non-K\"ahler compact complex manifolds.  Among these are the LVM manifolds (cf. \cite{lv},  \cite{meersseman}), the more general LVMB manifolds \cite{b,BM}
  and the
(closely related) moment angle manifolds  \cite{PU, Ish}. In \cite{thakur}, inspired by Loeb-Nicolau's construction \cite{ln}, Sankaran and the second-named author 
obtained a family of complex structures on $S(\mathcal{L}_1) \times S(\mathcal{L}_2)$, where $S(\mathcal{L}_i) \to X_i$, $i=1,2,$ is the smooth principal $S^1$-bundle associated to a holomorphic principal $\mathbb C^*$-bundle $\mathcal{L}_i \rightarrow X_i$.  Broadly speaking, our construction works in a more general setting compared to these constructions. However,  these other constructions use larger classes of foliations and consequently obtain a larger class of complex structures.

 By the description in \cite{zaffran}, an LVMB manifold satisfying  a certain rationality condition (K)  admits a Seifert principal fibration over an orbifold toric variety. If the base of the fibration is a nonsingular toric variety then the corresponding LVMB manifolds
may be recovered as a special case of our construction.

 It is possible to look at all LVMB manifolds that satisfy condition (K) as part of a more general phenomenon. It is shown in \cite{zaffran} that any such LVMB manifold can be identified to an orbit space $X/L$,  where $X$ is a complex manifold that admits a proper holomorphic action of $(\mathbb{C}^*)^{2m}$ with finite stabilizers, and $L$ is a torsion-free closed co-compact subgroup of $(\mathbb{ C}^*)^{2m}$. In Section \ref{proper}, we  consider the more general situation of any proper holomorphic action  of a complex linear algebraic group $G$ on a complex manifold $X$. If the rank of $G$ is greater than one, then $G$ has a nontrivial closed torsion-free
complex Lie subgroup $L$ which acts freely on $X$. This yields new, possibly non-K\"ahler, complex manifolds $X/L$. In particular, given an even dimensional effective complex analytic orbifold $V^{2n}$, we can take $X$ to be the manifold of frames of the holomorphic tangent (or cotangent) bundle of $V^{2n}$, and $G$ to be $GL(2n,\mathbb{C})$. Then, there exists $L$ such that $G/L$ is diffeomorphic to  
$ U(2n)$, and $X/L$ is a non-K\"ahler complex manifold. Analogous results are obtained in the case of odd dimensional Calabi-Yau orbifolds.   

There is another approach to constructing complex structure on  the total space of a smooth principal bundle with complex fiber and base, using a connection. This is particularly useful for torus principal bundles (see \cite{gp, ggp}). We discuss this approach for principal bundles with an even dimensional compact Lie group as structure group in Section \ref{ap}.

The study of special Hermitian structures on non-K\"ahler complex manifolds, such as SKT and CYT structures, have gained in importance as they play a key role in certain types of string theory (cf. \cite{GHR, Hul, Str, bbg}). They are also closely related to generalized K\"ahler geometry 
(cf. \cite{Hit, Gua, CG, FT, cavalcanti}). It is known that every compact Lie group $K$ of even dimension admits an SKT structure (cf. \cite{MS}). 
 Moreover,  the Calabi-Eckmann manifolds $S^1 \times S^1$, $S^1 \times S^3$ and $S^3 \times S^3$ admit SKT structure. The remaining Calabi-Eckmann manifolds, $S^{2m-1} \times S^{2n-1}$, $m,n >1$, do not admit SKT structures (cf. Example 5.17 \cite{cavalcanti}).  Recently, 
 non-existence of SKT structures on $SU(5)/T^2$ has been shown in \cite{fino}. The complex manifold $SU(5)/T^2$ is the total space of a holomorpic principal $T^2$-bundle over the flag manifold $SU(5)/T^4$.
 
 However, in the case of principal torus bundles with even dimensional fibers and first Chern class of type $(1,1)$, a number of interesting constructions and results were obtained in \cite{ggp} for both SKT and CYT structures. 
 We relate our construction of complex structures on principal bundles to these in Section \ref{shs}.   
 Observe that the existence of a complexification should have a bearing on the Dolbeault grading of  the characteristic classes. 
 
Let $K$ be an even dimensional compact connected Lie group.   If $E_K \to M$ admits a complexification $E_{K_{\mathbb{C}}} \rightarrow M$, then there exits a principal torus bundle $E_K \to X$ which also admits a complexification.  Here $X=E_{K_{\mathbb{C}}}/B \cong E_K/ T$ where $T$ is a maximal torus in $K$ and $B$ is a Borel subgroup of $K_{\mathbb{C}}$ that contains $T$.  We can choose the complex structure on
 $E_K$ so that it matches the one  constructed in \cite{ggp}. Hence the results obtained there may be applied to $E_K$. In addition, assume $K$ is an unitary, a special orthogonal or a compact symplectic group.   We obtain families of  SKT structures on  $K$ parametrized by an open subset of the space of K\"ahler structures on $K_{\mathbb{C}}/B$ (cf. Lemma \ref{grc}). 
 When $M$ is a  projective manifold, then  $X$ is also projective. In this case, we obtain sufficient  conditions for existence of SKT structure on $E_K$  in terms of   the characteristic classes of $E_K \rightarrow M$ (cf. Theorem \ref{pbc}). If $K$ is the unitary group, then the condition is that the square of the first Chern class equals the second Chern class. If $K$ is the special orthogonal or the compact symplectic group, the condition is that the corresponding first Pontryagin class  is zero.

  Let $K$ be a simply connected classical Lie group, namely the compact symplectic or special unitary group,  of even dimension.
  We study the Picard group of $E_K$  in Section \ref{picard}  when the base manifold $M$ is a simply-connected projective manifold with
  trivial $h^{0,2}$. We obtain  ${\rm Pic}(E_K) \cong {\rm Pic}(M) \oplus H^1(T, \mathcal{O}_T)$. 
 Under these assumptions,  we also have
  $H^i(E_K,\mathbb{Z}) \cong  H^i(M,\mathbb{Z}) $ for $i=1,2$. Furthermore, note that the Dolbeault cohomology groups of $K$ have been computed in \cite{pittie}. These, when used with the inequality 7.5 in Appendix 2 of \cite{Hir} which is derived from the Borel spectral sequence, provide some  upper bounds for $h^{p,q}(E_K)$.
  
  If the complex structure on $T$ is generic, then $T$ admits no non-constant meromorphic functions. Hence meromorphic functions on $E_K$ are pull-backs of meromorphic functions on $X$.
On the other hand, suppose the complex structure on $E_K$ is in accordance with  Section \ref{shs}, so that $T$ has the complex structure of a product of elliptic curves.  In this case $T$ has the highest possible transcendence degree. However, using the above description of the Picard group, we find that $E_K$ and $X$ still have isomorphic fields of meromorphic functions. See Theorem \ref{adim}.  It follows that $E_K$ is not a Moishezon manifold or an abstract algebraic variety in these cases.


\section{Complex structures on an even dimensional compact Lie group}\label{cs} 


 Samelson \cite{Sam} and Wang \cite{wang} proved the existence of a family of left invariant complex analytic structures on an even dimensional compact Lie group. A classification of such structures was given in \cite{pittie}.  Recently, the topic was revisited in \cite{LMN} where  complex analytic structures that are not invariant were also constructed.
 We  briefly review the left invariant complex structures 
 following \cite{pittie} and \cite{LMN} for the ease of reference later on.

 Let $K$ be an even dimensional compact  Lie group of rank $2r$. 
 Let $\mathfrak{k}$ be the Lie algebra of $K$. 
 Let $K_{\mathbb{C}}$ be the universal complexification of $K$ and $\mathfrak{k}_{\mathbb{C}} = \mathfrak{k} \otimes \mathbb C$ be its Lie algebra. Let $J$ be an  almost complex structure on  $\mathfrak{k}$. We denote its extension to $\mathfrak{k}_{\mathbb{C}}$ by the same symbol.  Let $\mathfrak{l}$ denote the $i$ or $-i$ eigenspace of $J$ on $\mathfrak{k}_{\mathbb{C}} $. 
  Then $J$ induces a left invariant  integrable complex structure 
on $K$ if and only if $\mathfrak{l}$ is involutive with respect to the Lie bracket.    
  
First assume $K$ is connected.  Then Proposition 2.2 of \cite{pittie} states that a left invariant complex structure on $K$  
is determined by a complex Lie subalgebra $\mathfrak{l}$ of $\mathfrak{k}_{\mathbb{C}} $ such \begin{equation} \label{subalgebra}
 \mathfrak{l} \bigcap \mathfrak{k} = 0 \mbox{ and } \mathfrak{k}_{\mathbb{C}} = \mathfrak{l} \oplus \bar{\mathfrak{l}}.
 \end{equation} 
  Pittie \cite{pittie} called a complex Lie subalgebra $\mathfrak{l}$ of $\mathfrak{k}_{\mathbb{C}} $ satisfying condition  \eqref{subalgebra} a Samelson subalgebra. 
  Our main interest is in the complex Lie subgroup $L$ of $K_{\mathbb{C}}$ corresponding to a Samelson subalgebra $\mathfrak{l}$. We assume that the Lie subalgebra $\mathfrak{l}$  defines vectors of type $(0,1)$ in $T_e(K) \otimes \mathbb C = \mathfrak{k}_{\mathbb{C}} $. This is
  necessary for the conformity of the complex structures on $K$ and $K_{\mathbb{C}}/L$ in Proposition \ref{semisimple} below.

  Fix a maximal torus $T \cong (S^1)^{2r}$ of $K$ and let  $B$ be a Borel subgroup of $G$ containing $T$. Let $\mathfrak{t}$ and $\mathfrak{b}$ be the Lie algebras of $T$ and $B$ respectively. Let $\mathfrak{h} = \mathfrak{t} \otimes \mathbb C$ be the Cartan subalgebra of $\mathfrak{g}$ containing $\mathfrak{t}$. 
Let  $\mathfrak{a}$ be a subalgebra of $\mathfrak{h}$ such that $ \mathfrak{h} = \mathfrak{a} \oplus \bar{\mathfrak{a}}$. Then 
  \begin{equation}\label{nilpotent}
 \mathfrak{l} := \mathfrak{a} \oplus  [\mathfrak{b}, \mathfrak{b}]
 \end{equation}  
 is a Samelson subalgebra, and every Samelson subalgebra is of this form for appropriate choices of $\mathfrak{t}$, $\mathfrak{b}$ and $\mathfrak{a}$ (cf. Corollary 2.5.1, \cite{pittie}). 
  
  For  fixed $\mathfrak{t}$ and $\mathfrak{b}$, the possible choices of $\mathfrak{a}$ are  described below.
   Fix a real basis for $\mathfrak{t}$ and identify $\mathfrak{h}$ with $\mathbb C^{2r}$ so that $\mathfrak{t}$ is included in $\mathfrak{h}$ as $i\mathbb R^{2r} \subset \mathbb C^{2r}$. Now consider a $\mathbb C$-linear map $\Lambda: \mathbb C^r  \rightarrow \mathfrak{h}$ given by
  $$z \mapsto A\cdot z$$
  where $A= (a_i^j)$ is a $(2r\times r)$-complex matrix. We denote the image of this linear map by $\mathfrak{a}$.  Let 
  $$ A_{\Lambda} :=\begin{pmatrix}
  \mbox{Re } a_1^1 & -\mbox{Im }a_1^1 & \cdots & \mbox{Re }a_1^r & -\mbox{Im }a_1^r\\
  \vdots    &      \vdots       &        &  \vdots          &  \vdots \\
  \mbox{Re } a_{2r}^1 & -\mbox{Im }a_{2r}^1 & \cdots & \mbox{Re }a_{2r}^r & -\mbox{Im }a_{2r}^r
  \end{pmatrix}
  $$ be the $(2r \times 2r)$-real matrix where  $ \mbox{Re } a_{i}^j$ and $\mbox{Im } a_{i}^j$ denote the real and imaginary parts of 
  $a_i^j$ respectively. 
  Note that  $\mathfrak{a} \bigcap \mathfrak{t} = 0$   if and only if  $A_{\Lambda}$ is non-degenerate (cf.  Lemma  3.1 \cite{LMN} ).
 In this case $\mathfrak{h} = \mathfrak{a} \oplus \bar{\mathfrak{a}}$. Moreover,
 any such decomposition of $\mathfrak{h}$
 corresponds to  a  $\mathbb C$-linear map $\Lambda: \mathbb C^r \rightarrow \mathfrak{h}$ with $\det A_{\Lambda} \neq 0$.

 Let $H \cong (\mathbb C^*)^{2r}$ be the Cartan subgroup of $G$ associated to  $\mathfrak{h}$. Let $\exp(\mathfrak{a})$ be the connected complex Lie subgroup of $H$ with Lie algebra $\mathfrak{a}$. Then $$\dim (\exp(\mathfrak{a}) \bigcap (S^1)^{2r}) = 0$$ as $\mathfrak{a} \bigcap \mathfrak{t} = 0$. In this case, it follows by Lemma 2.9 \cite{LMN} that $\exp(\mathfrak{a})$ is isomorphic to $\mathbb C^r$. By compactness of  $(S^1)^{2r}$,   $\exp(\mathfrak{a}) \bigcap (S^1)^{2r}$ is  finite. Since  $\exp(\mathfrak{a})$ is torsion-free, $\exp(\mathfrak{a}) \bigcap (S^1)^{2r} $ is trivial. Furthermore, $\exp(\mathfrak{a})$ is a closed subgroup of $H$ (cf. Lemma 3.1, \cite{LMN}).

Note that the subalgebra $ [\mathfrak{b},\mathfrak{b}]$ is an ideal in $\mathfrak{b}$.
 The connected complex Lie subgroup $U$ of $K_{\mathbb{C}}$ corresponding to  $[\mathfrak{b},\mathfrak{b}] $  is the unipotent radical of  $B$  and we have  the semidirect product decomposition,
  $B = H \cdot U$. Hence the connected complex Lie subgroup $L$ of $K_{\mathbb{C}}$ associated to the Samelson subalgebra
 $\mathfrak{l} = \mathfrak{a} \oplus  [\mathfrak{b},\mathfrak{b}]  $ is of the form $$L= \exp(\mathfrak{a})\cdot U \,.$$ As $U$ is torsion-free and contractible (cf. Proposition 8.2.1 \cite{springer}), it follows that $L$ is a contractible, closed and 
torsion-free subgroup of $K_{\mathbb{C}}$.

Let $\widetilde{K_{\mathbb{C}}}$ be the simply-connected complex Lie group with Lie algebra $\mathfrak{k}_{\mathbb{C}}$ and let $\widetilde{L}$ be the connected closed Lie subgroup of  $\widetilde{K_{\mathbb{C}}}$  to Lie algebra $\mathfrak{l}$. Then $\widetilde{L}$ is
a covering group of $L$. However $L$ is contractible. So   $\widetilde{L}$ is isomorphic to $L$.  The proof of  Proposition \ref{semisimple} below
for  connected $K$ is then evident  from Proposition 2.3 \cite{pittie} and the discussion below it.

 \begin{prop}\label{semisimple}
 	Let $K$ be an even dimensional  compact Lie group with universal complexification $K_{\mathbb{C}}$. Let $K$ be endowed with the left invariant complex structure define by a Samelson subalgebra $\mathfrak{l}$. Let $L$ be the connected complex Lie subgroup of $K_{\mathbb{C}}$ associated to $\mathfrak{l}$. Then $L$ is a torsion-free closed subgroup of $K_{\mathbb{C}}$ such that $K \bigcap L = 1$  and $KL = K_{\mathbb{C}}$.   Moreover, the  natural  map $K \rightarrow K_{\mathbb{C}}/L$ is a biholomorphism. \qed
 \end{prop}
 
 For a general compact Lie group $K$, let $K^0$ be the identity component. We identify the subgroup $L$  of $(K^0)_{\mathbb{C}}$ obtained above  with its image in $K_{\mathbb{C}}$.  Consider the commutative diagram (cf. Section 6, Chapter III of \cite{Bour}),

\begin{equation}
\begin{CD}
 1 @>>>K^0  @ > >> K @> >> K/K^0 @>>>1  \\
@V VV  @V VV @V VV @V \rm{id} VV @V VV\\
 1@>>> (K^0)_{\mathbb{C}} @ > i>>  K_{\mathbb{C}}@>>> K/K^0@ >>>1
\end{CD}
\end{equation}
  
 This shows that $K$ and $K_{\mathbb{C}}$ have the same number of connected components. For any component of $K_{\mathbb{C}}$, we take 
 an element  $g$ of the corresponding component of $K$. Then left translation by $g$ gives a diffeomorphism between $K^0$ (respectively, $(K^0)_\mathbb{C}$) and the said component of $K$ (respectively, $K_{\mathbb{C}}$). We obtain a complex structure on $K_{\mathbb{C}}$ using this  diffeomorphism. This complex structure is independent of the choice of $g$. Moreover, the left translation by $g$ commutes with the right action of $L$. This ensures the validity of Proposition \ref{semisimple} when $K$ has multiple components.  
  
 Note that  Proposition \ref{semisimple} shows that
 there is a  holomorphic principal $L$-bundle, $$L \hookrightarrow K_{\mathbb{C}} \rightarrow K.$$

 
\section{Complex structures on the total space of principal bundles}\label{mr}

Let $K$ be a compact connected Lie group of even dimension. Let $G$ be a complex Lie group that contains $K$  as a real Lie
subgroup.  Let $L$ be a
closed complex Lie subgroup of $G$ such that $K \bigcap L = 1_G$  and $KL = G$. Then the natural map $K \rightarrow G/L$ is a 
diffeomorphism and this induces a complex structure on $K$. 

In these notes, by $L$-foliation we mean a foliation with leaf $L$. Consider the 
  left $G$-invariant $L$-foliation on $G$, given by the right translation action of $L$ on $G$. 
  Note  that this foliation is transverse to $K$. 

 \begin{prop}\label{prop1} Let $K, G$ and $ L$ be as defined above. 
 Let $E_G \rightarrow M$ be a holomorphic principal $G$-bundle
 over a complex manifold $M$. Assume that $E_G \rightarrow M$ admits a smooth
 reduction of structure group to give a smooth principal $K$-bundle, $E_K \rightarrow M$. Then the total space 
  $E_K$ admits a complex structure such that $E_K \rightarrow M$ is a holomorphic 
  fiber bundle with fiber $K$. Further, if $G$ is abelian then $E_K \rightarrow M$ is a 
  holomorphic principal $K$-bundle.
 \end{prop}

 \begin{proof} Since the foliation on $G$ corresponding to the right action of $L $ is left $G$-invariant,
 it induces an $L$-foliation on $E_G$.  This foliation is
 complex analytic as $E_G$ is a holomorphic bundle. Similarly, there is an induced smooth $L$-foliation 
 on   $E_K \times_K G$.

 Let $i: E_K  \hookrightarrow  E_K \times_K G  $ be the
  inclusion map defined  by $i(e) = [(e,1_G)]$. Since 
 $E_K \rightarrow M$ is obtained by a smooth
 reduction of structure group from $E_G \rightarrow M$, there exists  a smooth isomorphism of
 principal $G$-bundles, $\rho: E_K \times_K G \to E_G$. Then $\rho \circ i$ is a smooth embedding
 of $E_K$ into $E_G$.

 The image of $i$ is transversal to the $L$-foliation in $E_K \times_K G$. 
 Since the map $\rho$ is $L$-equivariant, it 
 respects the fiberwise $L$-foliations on  $E_K \times_K G$ and $E_G$. Being a (fiberwise)
 diffeomorphism, $\rho$ preserves transversality.
 Therefore, the $L$-foliation in every fiber of $E_G$
 is transversal to the image  of corresponding fiber of $E_K$  under $\rho \circ i $.
 It follows that the $L$-foliation on $E_G$ is
 transversal to the image of $E_K$ under $ \rho \circ i$.  
 Moreover, $L$ is closed and $K\bigcap L = 1_G$. Therefore, the map $E_K \rightarrow E_G/L$ induced by 
 $\rho \circ i$ is a diffeomorphism. 
 Hence, the holomorphic fiber bundle structure on $E_G/L \rightarrow M$  can be pulled back to give a 
 holomorphic fiber bundle structure on $E_K \rightarrow M$. 

Finally, if $G$ is abelian, then $G/L$ inherits the structure of a complex Lie group. This induces a complex Lie group structure
on $K$.  With respect to this structure $E_K \to M$ is a holomorphic principal $K$-bundle.
 \end{proof}

Let $N$ be a complex linear algebraic group.
 Let $K$ be a maximal compact subgroup of $N$.  As $K$ is a smooth deformation retract of $N$ (cf. Theorem 3.4, Chapter 4, \cite{VGO}), a holomorphic principal $N$-bundle, $E_N \rightarrow M$, over a complex manifold $M$ admits a smooth reduction of the structure group from $N$ to $K$.

\begin{prop}\label{linear}
	Let $N$ be a complex linear algebraic group and let $K$ be a maximal compact  subgroup of $N$. Let $E_N \rightarrow M$ be a holomorphic principal $N$-bundle. Let $E_K \rightarrow M$ be a smooth principal $K$-bundle corresponding to the smooth reduction of the structure group from $N$ to $K$. Then  $E_K \rightarrow M$ admits a complexification.	
\end{prop}
\begin{proof} Let $G$ be a reductive Levi subgroup of $N$ containing $K$. Then $G$ is a maximal reductive subgroup of $N$ and
 $G \cong K_{\mathbb{C}}$. Further, the algebraic group $N$ is isomorphic to the semi-direct product $G \cdot U$, where  $U$ is the unipotent radical of $N$ (cf. Chapter 6, \cite{onishchik}). 

 Consider the right action of $U$ on $E_N$. The induced map $E_N/U \rightarrow  M$ defines a holomorphic principal $G$-bundle. Observe that the smooth principal $G$-bundles $E_K \times_K G \rightarrow M$ and $E_N/U \rightarrow M$ are smoothly  isomorphic. 
	\end{proof}
  
  We now state our main theorem.
	\begin{theorem}\label{main} 
	Let $K$ be an even dimensional compact  Lie group endowed with a left invariant complex structure.  
	 Let $E_K \rightarrow M$ be a smooth
 principal $K$-bundle over a complex manifold $M$, which admits a complexification. Then the total space 
  $E_K$ admits a complex structure such that $E_K \rightarrow M$ is a holomorphic fiber bundle with fiber $K$. Further, if $K$ is abelian then $E_K \rightarrow M$ has the structure of a holomorphic principal $K$-bundle. \end{theorem}
  
  \begin{proof} The result  follows by applying 
  Propositions \ref{prop1} and \ref{semisimple}. Note that if $K$ is abelian, then its  complexification is also abelian.
  \end{proof}
  
  In most of the cases the complex manifold $E_K$ is non-K\"ahler. First assume that $K$ is a non-abelain compact Lie group.  Then, it is shown in Example 2.10, \cite{monica} that $K$ is not symplectic and hence non-K\"ahler with respect to any complex structure.

  \begin{theorem}\label{nonkahler}
  	Assume that $K$ is a non-abelian compact  Lie group of even dimension. Then the complex manifold $E_K$ of Theorem \ref{main} 
	is non-K\"ahler. 
  \end{theorem} 

\begin{proof}
Any  fiber of the holomorphic fiber bundle $E_K \rightarrow M$ is a complex submanifold of the total space $E_K$. Each fiber is diffeomorphic to  $K$, and hence non-K\"ahler. The theorem  follows since a closed complex submanifold of a  K\" ahler manifold is K\" ahler. 
 \end{proof}
 
  The following corollary and example yield  new classes of  non-K\"ahler complex  manifolds.
  
   \begin{corollary}\label{unitary}
  	  	Let $E \rightarrow M$ be a holomorphic vector bundle of rank $2n$ over a complex manifold $M$.  Let $U(E) \rightarrow M$ denote the unitary frame bundle   associated to the vector bundle $E \rightarrow M $.  Then  $U(E)$ is a  non-K\"ahler complex manifold. Moreover,  $U(E) \rightarrow M$ has the structure of a holomorphic fiber bundle.
	\end{corollary}	

\begin{proof}	The frame bundle associated to the vector bundle $E \rightarrow M$  is a  holomorphic principal $GL(2n, \mathbb C)$-bundle. Then
the unitary frame bundle $U(E)\rightarrow M$ is a smooth principal $U(2n)$-bundle corresponding to the smooth reduction of structure group from $GL(2n,\mathbb C)$ to  $U(2n)$. As $GL(2n,\mathbb C)$ is the complexification of $U(2n)$, 	
by Theorem \ref{main}, the bundle $U(E) \rightarrow M$ has the structure of a holomorphic fiber bundle. The total space $U(E)$ is not K\"ahler by Theorem 
\ref{nonkahler}.
 \end{proof}

 \begin{example}\label{CY} Let $M$ be a complex $n$-dimensional Calabi-Yau manifold, where $n$ is odd. Then $M$ is a K\"ahler manifold that admits a non-vanishing holomorphic $n$-form $\Omega$. This implies that the holomorphic cotangent bundle $T^{\ast (1,0)}M$ of $M$ admits a holomorphic reduction of structure group to $SL(n,\mathbb{C})$.
 Since $SU(n)$ is a maximal compact subgroup of $SL(n,\mathbb{C})$, $T^{\ast (1,0)}M$ admits a smooth reduction of structure group to $SU(n)$. Note that $SU(n)$ is even dimensional when $n$ is odd. 
Let $G= SL(n,\mathbb{C})$ and $K = SU(n)$. Let   $E_G \rightarrow M$  be the bundle of precisely those  frames of $T^{\ast(1,0)}M$ the wedge product of whose components equals $\Omega$. Then $E_G \rightarrow M$ is a holomorphic principal $G$-bundle. Let  $E_K \rightarrow M$ be the associated principal $K$-bundle corresponding to the smooth reduction of structure group from $G$ to $K$.  
 Then by Theorems \ref{main} and \ref{nonkahler}, the total space $E_K$ is a non-K\"ahler  complex  manifold.  Note that $E_K$ is smoothly isomorphic to the bundle of special unitary frames of $T^{\ast(1,0)}M$.

By duality, the form $\Omega$  induces a holomorphic trivialization of the line bundle $\wedge^n T^{(1,0)}M $.  Therefore $T^{(1,0)}M $  admits holomorphic reduction of structure group to $SL(n, \mathbb{C})$. So, by a similar argument as above, the space of special unitary frames of 
 $T^{(1,0)}M$ admits the structure of a non-K\"ahler  complex  manifold. 
 
 In a similar fashion, if $M$ is a hyperk\"ahler manifold of complex dimension $4n$, then its tangent frame bundle admits a holomorphic reduction of structure group to $Sp(4n,\mathbb{C})$ and consequently a smooth reduction to $Sp(2n)$. The total space of the  smooth principal
  $Sp(2n)$-bundle is a non-K\"ahler complex manifold.  \qed \end{example} 
 
  \begin{example}\label{Stein} Let $M$ be a Stein manifold and $K$ a compact even dimensional Lie group. Given any smooth principal $K$-bundle $E_K \to M$, the associated principal
  $K_{\mathbb C}$-bundle $E_K \times_K K_{\mathbb C} \to M$ admits a unique holomorphic structure by Grauert's Oka-principle (cf. \cite{Gr}). Therefore, 
   $E_K$ admits an integrable complex structure which is non-K\"ahler if $K$ is non-abelian. This result can be further strengthened to the equivariant situation by applying the equivariant Oka-principle in \cite{HK}. Precisely, let $\Gamma$ be a compact Lie group. Let
    $E_K \to M$ be a smooth $\Gamma$-equivariant principal $K$-bundle such that $\Gamma$ acts by holomorphic transformations on $M$.
      Then $\Gamma$ acts smoothly, and hence holomorphically  on  
     $E_K \times_K K_{\mathbb C} \to M$   (cf. p. 2 of \cite{HI}). As the actions of $\Gamma$ and $L \subset K_{\mathbb C}$ commute, 
     $E_K \to M$ admits the structure of a $\Gamma$-equivariant holomorphic fiber bundle.
   \qed \end{example}

 When $K$ is an even dimensional compact torus,  we observe that the complex manifold $E_K$ is again non-K\"ahler  under certain conditions. In this case,  $K_{\mathbb{C}} = (\mathbb C^*)^{2r}$. Then the rank $2r$ vector bundle, corresponding to any holomorphic principal $(\mathbb C^*)^{2r}$-bundle, is the direct sum of $2r$ holomorphic line bundles, 
 $\mathcal{L}_1,  \ldots, \mathcal{L}_{2r}$. The following result follows from Proposition 11.3 \cite{hofer} and Corollary 11.4 \cite{hofer}.
 
 \begin{theorem}[H\"ofer]
 	Assume that $K$ is an even dimensional, compact torus and $M$ is a simply-connected compact complex manifold. If the Chern classes 
	$c_1(\mathcal{L}_1),  \ldots ,$ $ c_1(\mathcal{L}_{2r}) \in H^2(M, \mathbb R)$ are $\mathbb R$-linearly independent then the complex manifold $E_K$  of Theorem \ref{main} is not symplectic, and hence not K\"ahler. \qed
 \end{theorem} 
 
 In the case when $M$ is K\"ahler, Blanchard \cite{blanchard} (cf. Section 1.7, \cite{hofer}) gives the following necessary and sufficient condition.
 
 \begin{theorem}[Blanchard]\label{kbase} Assume that $K$ is an even dimensional, compact torus and $M$ is a compact K\"ahler manifold.  Then the complex manifold $E_K$  of Theorem \ref{main} is K\"ahler if and only if 
 $ c_1 ( \mathcal{L}_i ) \in  H^2(M, \mathbb R)  $ is zero for each $i$. \qed \end{theorem}

  	  If $K$ is an elliptic curve,  the following gives a sufficient condition without assuming the compact complex manifold $M$ 
 to be simply-connected or K\"ahler
	  (cf. Corollary 1, \cite{Vuletescu}).  
	   	
  	\begin{theorem}[Vuletescu]\label{kbase2}
  	Assume that $K$ is an elliptic curve and at least one of the Chern classes $c_1(\mathcal{L}_i)   \in  H^2(M, \mathbb R) $  is non-zero.
 Then the complex manifold $E_K$  of Theorem \ref{main} is  a non-K\"ahler complex manifold. \qed
  	\end{theorem}

  \begin{example}\label{toric} Consider a nonsingular  toric variety $M$ of complex dimension $n$. 
  Assume that the one dimensional cones in the fan $\Sigma$ of $M$ are generated by primitive integral vectors 
  $\rho_1, \ldots, \rho_k$. Let $R$ be the matrix $[\rho_1 \ldots \rho_k]$.
   Let $\mathcal{E}_j $ be the  line bundle  over $M$ corresponding to the $j$-th one dimensional cone.  
   Then any algebraic line bundle over $M$ is of the form $\bigotimes \mathcal{E}_j^{a_j} $ where each $a_j$ is an integer.
   
   The first Chern class of such a line bundle is given by
   $$c_1( \bigotimes \mathcal{E}_j^{a_j} ) = \sum a_j c_1 (\mathcal{E}_j) \,. $$ It is zero if and only if the vector $ (a_1, \ldots, a_k)$ belongs 
   to the row space of the matrix $R$.
   
   Consider an algebraic vector bundle $E \rightarrow M$ which is the direct sum of an even number, namely $2r$, of line bundles. 
   Let $K = (S^1)^{2r}$. Let $E_K \rightarrow M$ denote the smooth principal $K$-bundle obtained by reduction of structure group from 
  the principal holomorphic $(\mathbb{C}^*)^{2r}$-bundle associated to $E$.  The total space  $E_K$  of this bundle admits a family of
  complex analytic structures by Theorem \ref{main}.
  
   If  $M$ is projective, then Theorem \ref{kbase} gives a sufficient condition
  for the total space $E_K$ to be non-K\"ahler.  The moment angle manifold corresponding to the fan $\Sigma$ (cf. \cite{PU})
  is an example of  such an $E_K$. 
   \end{example}

    Now, we discuss  the case when the compact Lie group $K$ is odd dimensional.
 We refer to \cite{blair} for the definition of a normal almost contact structure (nacs) on a smooth manifold. We have the following result.
    
  \begin{corollary}\label{odd}
  	Let $K$ be an odd dimensional compact connected Lie group. Let $K \times S^1$ be endowed with a left invariant complex structure. If  a smooth principal $K$-bundle $E_K \rightarrow M$ over a complex manifold $M$  admits a complexification,  then the bundle  $E_K \times S^1 \rightarrow M$ has the structure
of a holomorphic fiber bundle with fibre $K \times S^1$. In particular, the space $E_K$ admits a normal almost contact structure.
  \end{corollary}

  \begin{proof} Let $E_{K_{\mathbb{C}}}  \rightarrow M$ denote the complexification of $E_K \rightarrow M$. 
Then  $E_{K_{\mathbb{C}} } \times \mathbb{C}^*   \rightarrow M$ is the complexification of the principal 
 $(K \times S^1)$-bundle  $E_K \times S^1 \rightarrow M$.  Since  $K \times S^1$ is even dimensional, the proof of the first part of the corollary
follows  from Theorem \ref{main}. The second part of the corollary follows from the fact that the inclusion map 
$E_K \to E_K \times \{0\} \subset E_K \times \mathbb R$ is an embedding of $E_K$ as an orientable real hypersurface of the complex manifold $E_K \times \mathbb R$ (cf.  Example 4.5.2 and Section 6.1 of \cite{blair}).
	\end{proof}
	
With regard to the above corollary, we note that a nacs endows a CR-structure of hypersurface type on $E_K$ (cf. Theorem 6.6 \cite{blair}).
	
\begin{remark}\label{oddcase} Let $K$ and $E_K$ be as in Corollary \ref{odd}.	If $K$ is non-abelian, then $E_K \times S^1$  is  a non-K\"ahler complex manifold. On the other hand, if $K$ is a compact torus, then Theorems  \ref{kbase} and 
\ref{kbase2} give some  conditions for $E_K \times S^1$ to be  non-K\"ahler.
\end{remark}

\section{Another approach to complex structures on $E_K$}\label{ap} 

There is another approach to constructing complex structures on the total space of a principal bundle using an $(1,0)$-connection with 
$(1,1)$-curvature. This was used very profitably in \cite{ggp} for principal torus bundles. Indeed, we will make use of their approach and results
 in the following section. Here we give an analogous construction for more general structure groups.  
We thank the referee for drawing our attention to this possibility. 

However, there is a significant difference between the torus case and the general case. 
In the torus case, the group operations  of $K$  are holomorphic with respect to the complex structure induced by a Samelson algebra as 
the condition $[JA,B]= J[A,B] $ is trivially satisfied. This does not hold in the general case.  

\begin{lemma} Let $K$ be endowed with a complex structure as in  Proposition \ref{semisimple}. Suppose $E_K \rightarrow M $ is a smooth principal $K$-bundle over a complex manifold $M$. Assume that $E_K \rightarrow M$ is endowed with an almost complex fiber bundle structure that agrees with the integrable complex structures on $M$ and $K$.
 If $E_K \rightarrow M $ admits a smooth $(1,0)$-connection with $(1,1)$-curvature, then the almost complex structure on $E_K$ is integrable, 
 and $E_K \rightarrow M $ admits the structure of a 
holomorphic fiber bundle.     
\end{lemma}

\begin{proof}
The argument is similar to Lemma 1 of \cite{ggp} .

If $V,\, W$ are vertical $(1,0)$ vector fields on $E_K$, then it is clear that $[V,W]$ is $(1,0)$ since $J$ is integrable on fibers.

 Now suppose 
$A^h, \, B^h$ are horizontal lifts of $(1,0)$ vector fields $A,\, B$ on $M$.  
Then, the curvature $\Omega(A, B) = - \omega ([A^h, B^h])$, where $\omega$ denotes the connection $1$-form.  Since 
$\Omega$ is of type $(1,1)$,  $\Omega(A,B) = 0$. Hence  $[A^h, B^h] $ is horizontal, implying that  $[A^h, B^h]  = [A,B]^h$.  However, $[A,B]^h$ is $(1,0)$ 
as $[A,B]$ and the connection are $(1,0)$. Therefore,
$[A^h, B^h]$ is $(1,0)$. 

 Finally, if $V$ is a vertical $(1,0)$ vector field and $A^h$ is the horizontal lift of a $(1,0)$ vector field $A$ on $M$, then 
$[V, A^h] = 0$ (and therefore (1,0)), as horizontal vector fields are preserved by the action of the structure group $K$.  This shows that 
$(TE_K)^{(1,0)}$ is involutive with respect to the Lie bracket. Hence $J$ is integrable.  It follows immediately that $E_K \rightarrow M$ is a holomorphic fiber 
bundle.
\end{proof}

Given  a
 holomorphic principal bundle $E_{K_{\mathbb{C}}} \rightarrow M$, consider a natural embedding of $E_K $ in $E_{K_{\mathbb{C}}}$ that corresponds to the inclusion $K \hookrightarrow K_{\mathbb{C}}$.
 A well-known result of Singer \cite{Singer} says that there exists a  smooth $(1,0)$-connection $\nabla$ with $(1,1)$-curvature on  
 $E_{K_{\mathbb{C}}} \rightarrow M$  which  restricts  to a  connection on $E_K \rightarrow M$. This restriction is 
 $(1,0)$ with $(1,1)$-curvature if the complex structure on $K_{\mathbb{C}}$ restricts to an invariant almost complex structure on $K$. This holds when  $K$ is a torus for suitable choice of complex structure: Let $J$ be 
 an integrable complex structure  on $K$ determined by a Samelson algebra. The scalar extension of $J$
 determines an integrable complex structure on $K_{\mathbb{C}}$. However, as remarked earlier,
the group operations on  $K_{\mathbb{C}}$ may not be holomorphic with respect to  $J$ when $K$ is not a torus.

\section{SKT structures on the total space of principal bundles}\label{shs}

Let  $K$ be a connected compact Lie group. Let  $T$, $H$, $B$,  $\exp(\mathfrak{a})$, $U$ and $L$ be the subgroups of 
$K_{\mathbb{C}}$ as in Section \ref{cs}. For notational convenience, we denote $K_{\mathbb{C}}$ by $G$ in this section.

Let $E_K \to M$ be a smooth principal $K$-bundle that admits a complexification $E_G \to M$. The Borel subgroup
$B=HU$ acts holomorphically on the total space $E_G$.
As $U$ is normal in $B$, the  subgroup $H$ acts holomorphically
on $E_H:= E_G /U$. Thus we have a holomorphic principal $H$-bundle $E_H \to E_G/B$.  We denote the complex 
manifold $E_G/B$ by $X$.

The quotient of the action of the subgroup $\exp(\mathfrak{a})$ of $H$  on $E_H$  is isomorphic to $E_G/L$. As
$H= \exp(\mathfrak{a})T$ is abelian, $E_G/L$ admits a holomorphic action of $T$. We identify $E_K$ with $E_G/L$ as
complex manifolds as in Section \ref{mr}. 

The above has two consequences: Firstly, the quotient of $E_K$ by a  finite subgroup of 
$T$ admits a complex structure. This produces more examples of non-K\"ahler complex manifolds. 
Secondly, the smooth principal $T$ bundle $E_K \to X$ admits a complexification $E_H \to X$. This is related to the existence of 
special Hermitian structures on $E_K$ as described below.  

The Bismut or KT connection (cf. \cite{Bis, ggp}) determined by a Hermitian metric $g$ on a complex manifold  is the Hermitian connection 
with skew symmetric torsion. The KT connection is called strong (or SKT) if  the  fundamental form of $g$ is $dd^c$-closed. 
On the other hand, it is called (locally) CYT if the restricted holonomy is in $SU$. It is known that every compact Lie group $K$ of even dimension admits an SKT structure (cf. \cite{MS}). 
 Moreover,  the Calabi-Eckmann manifolds $S^1 \times S^1$, $S^1 \times S^3$ and $S^3 \times S^3$ admit SKT structure. The remaining Calabi-Eckmann manifolds, $S^{2m-1} \times S^{2n-1}$, $m,n >1$, do not admit SKT structures (cf. Example 5.17 \cite{cavalcanti}). In the case of principal torus bundles with even dimensional fibers, a number of interesting constructions and results were obtained in \cite{ggp} for both SKT and CYT structures.  We will now relate our construction to their results. 
 
As noted above, we may consider $E_K$ as the total space of the principal $K$-bundle $E_K \to M$ or the principal $T$-bundle 
$E_K \to X$, which admit complexifications $E_G \to M$ and $E_H \to X$ respectively. 
    It is to be noted that the complex structures on $E_K$ derived from the two complexifications, namely, the  
identifications of $E_K$ with $E_G/L$ 
and $(E_G/U)/ \exp(\mathfrak{a})$, coincide.

  Let $J$
denote the left invariant complex structure on $T$, or equivalently, the almost complex structure on $\mathfrak{t}$,
 corresponding to a Samelson subalgebra $\mathfrak{a}$ of $\mathfrak{h}$.  Denote the integrable complex structure on $E_G$ as well as the induced complex structure on  $E_H=E_G/U$ by $I$.  The latter  matches with the bi-invariant complex structure of 
 $ H $, inherited as a subgroup of $G$, on the fibers of $E_H \to X$.

Then, as Singer  explains in \cite{Singer}, there  exists a  smooth $(1,0)$ connection $\nabla$ on $E_H$ whose horizontal 
spaces are $I$-stable
and whose curvature is $(1,1)$. 

 Consider the bundle
  $E_K \rightarrow X $ as the quotient of $E_H \rightarrow X$ by $\exp(\mathfrak{a})$.  The connection $\nabla$ induces a  connection  $\nabla'$ on the smooth principal bundle $E_K \rightarrow X $ whose connection  $1$-form will be denoted by  $\theta$.
    The horizontal spaces of $\theta$ are complex with respect to the  complex structure on $E_K \rightarrow X$ induced by the 
    identification with quotient of $E_H \rightarrow X$.  The holomorphic projection map $\pi: E_K \rightarrow X$ induces complex analytic isomorphisms between the horizontal spaces 
  of $\nabla'$ and the corresponding tangent spaces of $X$.  
  Also, $\theta$ is $(1,0)$ and has $(1,1)$ curvature with respect to the induced  complex structure (cf. proposition 6.1, \cite{KN}).
   By Section \ref{cs},  the induced  complex structure matches the complex structure $J$ in the vertical direction.  
The connection $\theta$ being $(1,0)$ means that 
\begin{equation}\label{jj} \theta (J\xi ) = J \theta (\xi) \end{equation}
for any vertical tangent vector $\xi \in TE_K$.

Consider a product decomposition of the group $T= \prod_{j=1}^{2r} S^1$. (For later use, this should correspond to the diagonal blocks of
the standard linear representation of $T$ as a subgroup of $G$ when $K$ is a classical Lie group other than SU(n).) 
 Let $\pi_j: \mathfrak{t} \rightarrow \mathbb{R}(\partial_{t_{j}})$ denote the corresponding 
$j$-th coordinate projection at the level of Lie algebras.   
Define $$\theta_{j} := \pi_j \circ \theta  \,.$$

Now, we choose a special kind of complex structure on $E_K$ that allows  decomposition of the complex torus $T$ into a product of 
elliptic curves.
Since $T$ is abelian, any almost complex structure on $\mathfrak{t}$ corresponds to a Samelson subalgebra. We choose the complex 
structure $J$ on $\mathfrak{t}$ so that $J (\partial_{t_{2j-1}}) = \partial_{t_{2j}}$.  Then \eqref{jj} implies that 
$$J\theta_{2j-1} =\theta_{2j}\,. $$ 

Thus the connection $\theta$ satisfies all the properties  of  the connection constructed in Lemma 1 of \cite{ggp}. 
Therefore, the complex structure on $E_K$ constructed there matches the complex structure from the quotient construction, with the choice of
$J$ as above. 

 More generally,
Lemma 1 of \cite{ggp} considers  principal torus bundles  with even dimensional fiber and $(1,1)$-curvature. If the base manifold is 
compact, K\"ahler and the curvature is integral, then by the Lefschetz theorem on $(1,1)$ classes, any such bundle can be obtained from  a direct sum of holomorphic line bundles by the quotient construction.   In fact, most applications in \cite{ggp} assume these conditions.

Following \cite{ggp},
given any Hermitian metric $g_X$ on the base $X$,  define a Hermitian metric $g_E$ on the total space of $E_K$, by 
$$ g_E := \pi^* g_X + \sum_{j=1}^{2r} (\theta_j \otimes \theta_j)\,. $$ 
Consider the KT connection on $E_K$ induced by $g_{E}$.  Propositions 9 and 11 of \cite{ggp} give some relatively simple conditions for this connection to be 
(locally) CYT. 
In the sequel, we will focus on SKT structures. 
By Proposition 14 of \cite{ggp},  
 the above connection is SKT if and only if 
 \begin{equation}\label{skt} \sum_{j=1}^{2r} \omega_j \wedge \omega_j =dd^c F_X \,, \end{equation} 
where $\pi^* \omega_j = d\theta_j$, $d^c= I^{-1}dI$ and $F_X$ is the fundamental form of $g_X$.
We apply this to obtain new families of SKT structures on $K$ when $K$ is an unitary, special orthogonal or compact symplectic 
group of even dimension. 

\begin{lemma}\label{grc} Let $K$ be an even dimensional unitary, special orthogonal or compact symplectic group. Then $K$ admits a
family of SKT structures parametrized by an open subset of the space of K\"ahler metrics on $K/T$ 
where $T$ is a maximal torus in $K$. \end{lemma}

\begin{proof} Consider $K$ as a principal $T$-bundle over the K\"ahler manifold $X=  G/B \cong  K/T $. It is well known (cf. \cite{borel})  that 
$$H^{*}(K/T) \cong \frac{H^{*}(BT)}{( H^{+}(BT))^{W(K)}} \,$$
where the $H^{+}$ denotes the cohomology in positive degrees. The image of the generators of $H^{+}(BT)$  in $H^{*}(K/T)$ 
can be identified with certain roots of $K$ or equivalently with the
Chern classes $[w_i]$, up to sign, of the corresponding $S^1$-principal bundles over $X$. The invariants of the adjoint action of the normalizer $N_T(K)$  (equivalently, the invariants of the action of the Weyl group
$W(K)$)  on these roots are described in \cite{KN2}. It is then easily verified that the class of $\sum w_j\wedge w_j $ in $H^{*}(X)$ is $0$.

Thus the (2,2)-form $\sum w_j\wedge w_j $ is $d$-exact and $d$-closed. Denote  the complex structure on $X$ by 
 $I$ without confusion. Since $I$ preserves $(2,2)$-forms, it follows that $\sum w_j\wedge w_j $
is $d^c$-closed. Hence by the $dd^c$-lemma, there exists a real $(1,1)$-form $\alpha$ on $X$ such that $\sum w_i \wedge w_i = dd^c\alpha $.
Then, as in Theorem 15 of \cite{ggp},  we choose an appropriate multiple $\beta$ of a K\"ahler form on $X$ such that
$$\min_{p\in X}(\min_{|| Y || =1} \beta_p(Y,IY) ) > - \min_{p\in X}(\min_{|| Y || =1} \alpha_p(Y,IY) )\,.$$ 
Then the positive definite form $\alpha+\beta$ defines a Hermitian metric $g_X$ on $X$ which satisfies \eqref{skt}. \end{proof}
   
  \begin{theorem}\label{pbc} Let $K$ be an even dimensional unitary, special orthogonal or compact symplectic group. Suppose $E_K \to M$ is a principal $K$-bundle over a projective manifold $M$, which admits a complexification $E_G \rightarrow M$.  Then $E_K$ admits an SKT structure if
  \begin{itemize} \item $c_1^2(E_K \to M) = c_2(E_K \to M)$ when $K$ is unitary,
\item $p_1(E_K \to M)= 0$ when $K$ is either special orthogonal or compact symplectic.
  \end{itemize} 
Here $c_i$ and $p_i$ refer to the appropriate Chern and Pontryagin classes.   \end{theorem} 
   
   \begin{proof} Note $X= E_G/B$ is the total space of a holomorphic bundle over $M$ with the flag variety $G/B$ as fiber. The flag variety $G/B$ is Fano. Recall that there is a natural  quotient map $p: X \to M$.   Then the tensor product of the relative anti-canonical bundle over $X$ with the pullback of an  ample bundle over $M$ is ample.  Therefore, $X$ is projective.
   
 First, consider  the case that $K$ is unitary.   Then the pull-back bundle of $E_K \rightarrow M $ under $p$ can be 
 obtained from $E_K \to X$ by extension of structure group from $T$ to $K$.
  Thus the total Chern class of the $K$-bundle $E_K \rightarrow M$ pulls back to the total Chern class of the $T$-bundle $E_K \rightarrow X$  under  $p$. This implies that 
 $$\left[ \sum w_i\wedge w_i \right] = p^{*} (c_1^2 -c_2)(E_K \to M) \,.$$ Then, by similar arguments as in Lemma \ref{grc}, we obtain families of SKT structures
 on $E_K$ if  $c_1^2 = c_2$ for the $K$-bundle $E_K \to M$. 
 
 Similarly, if $K$ is a special orthogonal group or a compact symplectic group, $E_K$ admits SKT structures if the first Pontryagin class of $E_K \to M$ (cf.  Sections 9, 10 of \cite{BH}) is  zero.    \end{proof}
 
The following result is an immediate corollary of Theorem \ref{pbc}.
 
 \begin{corollary}\label{flat}  Let $K$ be an even dimensional unitary, special orthogonal or compact symplectic group. Suppose $E_K \to M$ is a smooth
 flat   principal 
 $K$-bundle over a projective manifold $M$. Then $E_K$ admits an SKT structure. \end{corollary}


 \section{Picard group of $E_K$}\label{picard}
 

 In this section we assume that $K$ is a simply connected compact  classical Lie group of even dimension. We also assume that the base manifold $M$ is a nonsingular projective toric variety. These assumptions are made to put adequate restrictions on cohomology groups. In fact, we can assume $M$ to be any simply-connected projective manifold with $h^{0,2}(M ) = 0$. Our calculations follow the general line of argument given in \cite{BMT}. For the results of this section, we do not require the complex torus $T$ to be a product of
 elliptic curves. 
 
Consider the principal torus bundle $T \hookrightarrow E_K \to X$, where $X = E_{K_{\mathbb{C}}}/B$,
 as in the last section. As it admits a complexification $E_H \to X$, by Theorem \ref{main} it is a holomorphic principal torus bundle.  
Therefore we may apply the Borel spectral sequence  to it. 
We obtain the following result. 

\begin{prop}\label{borelss} 
Let $T \stackrel{i}\hookrightarrow  E_K \stackrel{\phi}\rightarrow X $ be the holomorphic torus bundle described above. Then we have the following short exact sequence,
$$  0 \longrightarrow H^1(X, \mathcal{O}_X) \stackrel{\phi^*}\longrightarrow H^1(E_K, \mathcal{O}_{E_K}) \stackrel{i^*}\longrightarrow 
H^1(T, \mathcal{O}_T) \longrightarrow 0 \, ,$$ 
where $\phi^*$ and $i^*$ are the induced homomorphisms.
\end{prop}

\begin{proof} Consider the Borel spectral sequence (cf. Appendix 2 of \cite{Hir} ) associated to $T \stackrel{i}\hookrightarrow  E_K \stackrel{\phi}\rightarrow X $, for the trivial line bundle $\mathcal{O}_X$ over $X$. Since the structure group $T$ is connected,
we have 
$$ ^{p,q}E_2^{s,t} = \sum_i H^{i,s-i}(X, \mathcal{O}_X) \otimes H^{p-i, q-s+i} (T) \,,$$
where $p+q = s+t$.
Then, we have 
\begin{equation}\label{bss1}
\begin{array}{llll}
^{0,1}E_2^{1,0} &  \stackrel{d_2}\longrightarrow &   ^{0,2}E_2^{3,-1}  & = 0  \\
\\
^{0,1}E_2^{0,1} &  \stackrel{d_2}\longrightarrow &   ^{0,2}E_2^{2,0}  & =  H^{0,2} (X, \mathcal{O}_X) \,. \end{array}
\end{equation}
Recall that $X$ is the total space of a holomorphic fiber bundle over a projective toric manifold $M$ with a flag variety 
$K_{\mathbb{C}}/B = K/T$ as fiber.  
Applying the Borel spectral sequence, more precisely the inequality 7.5 in Appendix 2 of \cite{Hir},  to the bundle $ K/T \hookrightarrow X \rightarrow M $, we obtain
\begin{equation}\label{bss2} h^{0,2}(X) \le \sum_{i+j=2} h^{0,i}(M) \,h^{0,j}(K/T) \,. \end{equation}
It follows from Proposition 3.8 of \cite{pittie} that $h^{0,j} (K/T) = 0$ unless $j=0$. A similar result is well-known for projective toric manifolds. Thus from \eqref{bss2} we have, $h^{0,2}(X) = 0$. This readily implies that $H^{0,2} (X, \mathcal{O}_X) = 0$.

Since no element of $ ^{0,1}E_r^{1,0} $ and $ ^{0,1}E_r^{0,1} $ is $d_r$-boundary for $r\ge 2$, we have 
$$^{0,1}E_2^{1,0} \, = \,  ^{0,1}E_{\infty}^{1,0} \quad {\rm and } \quad   ^{0,1}E_2^{0,1} \, = \,  ^{0,1}E_{\infty}^{0,1} \,.$$
We have a filtration controlled by degree of  base component, 
$$ H^1(E_K, \mathcal{O}_{E_K}) = D^0 \supset D^1 \supset 0 \,, $$
where $D^1 = ^{0,1}E_{\infty}^{1,0} $ and $D^0/D^1= ^{0,1}E_{\infty}^{0,1} $. The associated graded object is 
$$ {\rm Gr} H^1(E_K, \mathcal{O}_{E_K}) =   \, ^{0,1}E_{\infty}^{1,0} \, \oplus \, ^{0,1}E_{\infty}^{0,1} \,.$$
Thus, the natural homomorphism 

$$ \phi^* : H^1(X, \mathcal{O}_X) = \, ^{0,1}E_2^{1,0} \longrightarrow \, ^{0,1}E_{\infty}^{1,0} = D^1 \subset H^1(E_K, \mathcal{O}_{E_K}) $$
is injective, and the natural homomorphism 

$$ i^* :   H^1(E_K, \mathcal{O}_{E_K}) = D^0 \longrightarrow D^0/D^1 \,= \,  ^{0,1}E_{\infty}^{0,1} \, = \, ^{0,1}E_{2}^{0,1}= H^1(T, \mathcal{O}_T)$$
is surjective. Therefore, we have the required exact sequence 
$$  0 \longrightarrow H^1(X, \mathcal{O}_X) \stackrel{\phi^*}\longrightarrow H^1(E_K, \mathcal{O}_{E_K}) \stackrel{i^*}\longrightarrow 
H^1(T, \mathcal{O}_T) \longrightarrow 0 \, .$$ \end{proof}

\begin{lemma} \label{sss0} Let $p: X \to M$ be the projection map. Then the homomorphism 
$$p^* : H^1(M, \mathcal{O}_{M})   \to H^1(X, \mathcal{O}_{X})  $$ is an isomorphism.
\end{lemma}

\begin{proof} The  map $p$ corresponds to a fiber bundle with the flag variety $K_{\mathbb{C}}/B$  as fiber. The flag variety  is simply connected. Therefore the long exact sequence for homotopy groups of $p$ implies that the homomorphism $\pi_1(X) \to \pi_1(M) $ induced by $p$ is an isomorphism. Hence $$p^*: H^1(M, \mathbb{Q}) \to H^1(X, \mathbb{Q})$$ is an isomorphism. The lemma follows since $X$ and $M$ are both
 K\"ahler.
\end{proof}

\begin{lemma}\label{sss1}
Let $K$ be a simply connected compact  classical Lie group of even dimension. 
Let $q: E_K \to M$ be the projection map. Then the  pullback homomorphism
$$ q^{\ast} : H^i(M\,, \, \mathbb{Z}) \rightarrow H^i(E_K, \mathbb{Z})\,  $$ is an isomorphism for $i=1,2$.
\end{lemma}

\begin{proof} We will verify the $i=2$ case. The other case is simpler.

Observe that $M$ is simply connected. Consider the Serre spectral sequence associated to the fiber bundle 
$$ K \hookrightarrow E_K \to M \,.  $$ Note that $H^{2}(K, \mathbb{Z})= 0$ and $H^{1}(M, \mathbb{Z})= 0$. Consequently, we have 

$$\begin{array}{l}  E_2^{0,2} = H^{0}(M, \mathbb{Z}) \otimes H^{2}(K, \mathbb{Z}) =0 \, , \\  
  \\
E_2^{1,1} = H^{1}(M, \mathbb{Z}) \otimes H^{1}(K, \mathbb{Z}) = 0 \, , \\ 
 \\
E_2^{2,0} = H^{2}(M, \mathbb{Z}) \otimes H^{0}(K, \mathbb{Z}) =   H^{2}(M, \mathbb{Z})  \, . \end{array} $$
\smallskip

\noindent Moreover, since $H^1(K, \mathbb{Z})=0$, $E_2^{0,1} = H^{0}(M, \mathbb{Z}) \otimes H^{1}(K, \mathbb{Z})   = 0$. Therefore,
 $$ d_2 : E_2^{0,1}  \rightarrow E_2^{2,0} $$ is the zero map. This and the above imply that
 
 $$ E_{\infty}^{0,2} = 0 \, , \, E_{\infty}^{1,1} = 0 \, , \, \rm{and} \,   E_{\infty}^{2,0} = H^2(M, \mathbb{Z}) \,  .$$ It follows that 
 
 $$q^{*}: H^2(M, \mathbb{Z}) = E_{\infty}^{2,0} \rightarrow  H^2(E_K, \mathbb{Z}) $$ is an isomorphism.
\end{proof}

\begin{theorem}\label{pic2} Let $K$ be a special unitary or compact symplectic Lie group of even dimension. Then
$$  {\rm Pic}(E_K) \cong {\rm Pic} (M) \bigoplus H^{1}(T,O_T)$$ 

\end{theorem}

\begin{proof} We have the following commutative diagram whose horizontal rows are exact and whose vertical arrows are induced by the map $q: E_K \to M $.

\scriptsize
\begin{equation}\label{cd}
\begin{CD}
   H^1(M,\mathbb{Z}) @> >> H^1(M,\mathcal{O}_M) @>>> {\rm Pic}(M) @>>> H^2(M,\mathbb{Z}) @>>> H^2(M, \mathcal{O}_M) \\
@V  q^*  VV  @V  q^*  VV  @V q^*VV  @V q^*VV  @V q^*VV  \\
 H^1(E_K,\mathbb{Z}) @> >> H^1(E_K,\mathcal{O}_{E_K}) @>>> {\rm Pic}(E_K) @>>> H^2 (E_K,\mathbb{Z}) @>>> H^2(E_K, \mathcal{O}_{E_K}) \end{CD}
\end{equation}
\normalsize

By Lemma \ref{sss1}, the first and fourth vertical arrows are isomorphisms. By Proposition \ref{borelss}  and Lemma \ref{sss0}, the second vertical arrow  is injective. Moreover as $ H^1(M,\mathbb{Z}) = 0 = H^1(E_K,\mathbb{Z})$, the second horizontal arrow in each row is an injection. Then by 
diagram chasing, it follows that $q^* : {\rm Pic}(M) \longrightarrow  {\rm Pic}(E_K)$ is injective.

Since $H^1(M,\mathcal{O}_M) = 0 = H^2(M,\mathcal{O}_M) $, the map ${\rm Pic}(M) \longrightarrow H^2(M,\mathbb{Z}) $ is an isomorphism. Note that the composition of this map with $q^*:  H^2(M,\mathbb{Z}) \longrightarrow H^2(E_K,\mathbb{Z})$ is also an isomorphism. Therefore, by commutativity of the diagram, the map $ {\rm Pic}(E_K) \longrightarrow H^2(E_K,\mathbb{Z})$ is surjective. Using the isomorphism between 
${\rm Pic}(M)$ and $H^2(E_K,\mathbb{Z})$, we have a split exact sequence 
\begin{equation}\label{sespic}
 0 \longrightarrow  H^1(E_K,\mathcal{O}_{E_K})  \longrightarrow {\rm Pic}(E_K) \longrightarrow {\rm Pic}(M) \longrightarrow 0 \,. 
 \end{equation}
By Proposition \ref{borelss} $ H^1(E_K,\mathcal{O}_{E_K}) \cong  H^{1}(T,O_T)$, since  by Lemma \ref{sss0}
 $H^1(X,\mathcal{O}_X) \cong H^1(M, \mathcal{O}_M)  $, and $H^{1}(M ,O_M)=0$.  This concludes the proof.
 \end{proof}


\section{Meromorphic functions on $E_K$}\label{mero}


Recall the holomorphic torus bundle $T \stackrel{i}\hookrightarrow  E_K \stackrel{\phi}\rightarrow X $ from the last section. If the complex structure on $T$ is generic, then $T$ admits no non-constant meromorphic functions. Hence meromorphic functions on $E_K$ are pull-backs of 
meromorphic functions on $X$.

On the other hand suppose the complex structure on $E_K$ is in accordance with  Section \ref{shs}, so that $T$ has the complex structure of a product of elliptic curves.  In this case $T$ has the highest possible transcendence degree. However, we find that $E_K$ and $X$ may still have isomorphic fields of meromorphic functions. 

\begin{theorem}\label{adim} Asume that the complex structure on $E_K$ is such that $T$ is isomorphic to a product of elliptic curves.
Assume that $M$ is a simply-connected projective manifold with $h^{0,2}(M ) = 0$. Furthermore, assume that $K$ is  a special unitary or compact symplectic Lie group of even dimension. Then the algebraic dimension of $E_K$ equals the complex dimension of $X$.
\end{theorem}

\begin{proof}
Recall from Proposition \ref{borelss}  that the restriction map induces a homomorphism $$i^* : H^{1}(E_K, \mathcal{O}_{E_K}) 
\longrightarrow H^1(T, \mathcal{O}_T )\, .$$ This is in fact an isomorphism as $H^1(X, \mathcal{O}_X)= 0$. 
 We have a commuting diagram 
\begin{equation}\label{cd2}
\begin{CD}
H^1(E_K,\mathcal{O}_{E_K}) @>  \exp >> {\rm Pic}(E_K) \\
 @V  i^*  V \cong V  @V  i^*  VV \\
 H^1(T,\mathcal{O}_{T}) @> \exp >> {\rm Pic}(T) 
\end{CD}
\end{equation}
where the horizontal arrows are induced by exponentiation. The top horizontal map is injective by \eqref{sespic}. 

Every nontrivial meromorphic function $f$ on $E_K$  can be expressed 
as the ratio of two holomorphic sections of a suitable line bundle $\mathcal{L} $ over $E_K$. 
The restriction of $f$ to $T$ may be expressed as the ratio of two sections of  $i^*(\mathcal{L})$. However, using
\eqref{cd2} and the split exact sequence \eqref{sespic},
we observe that  $i^*(\mathcal{L})$ is  in the image of $\exp: H^1(T,\mathcal{O}_{T}) \to {\rm Pic}(T) $. Therefore,
the first Chern class  of  $i^*(\mathcal{L})$ is zero. Consider the torus $T$ as a Cartesian product of elliptic curves 
$T = \prod_{j=1}^{r} C_j$. Let $ \alpha_j : C_j \to T  $ be any inclusion map compatible with this product decomposition.
 It follows that the bundle $(\alpha_j)^* i^* {\mathcal{L}}$ on $C_j$
 has degree zero. Therefore, using Riemann-Roch, the linear system of $(\alpha_j)^* i^* {\mathcal{L}}$ has dimension at most one.
  However, this  bundle has a nontrivial holomorphic section by construction. Therefore, the dimension of the linear system is one.
 Hence, the ratio of any two sections of this bundle is 
 constant. Therefore, the restriction of $f$ to any image $\alpha_j(C_j)$ is constant. This implies that the restriction of $f$  to $T$ is constant. We conclude that 
 a meromorphic function on $E_K$ is constant on every fiber of $\phi: E_K \to X$. Therefore, the field of meromorphic functions on $E_K$ consists 
 only of pullbacks of meromorphic functions on $X$. Consequently, the algebraic dimension of $E_K$ equals that of $X$. However, $X$ is  a complex algebraic variety. So its algebraic dimension equals its complex dimension. This completes the proof.
\end{proof}

\begin{remark} A similar argument as above shows that the meromorphic functions on the generalized Calabi-Eckman manifolds
 $S= S_{\lambda}^{\xi_1, \xi_2}$ studied in \cite{BMT} are constant along the torus fibers of the fibration  $T \hookrightarrow S \rightarrow Y$.  
\end{remark}


	\section{Proper action of a complex linear algebraic group}\label{proper}
	
  
  Consider any proper holomorphic right action of  a complex Lie group $N$ on a complex manifold $X$. A stabilizer $N_x, x \in X$, of this action is a compact Lie subgroup of $N$. 
 Suppose $L$ is a torsion-free closed complex Lie subgroup of $N$. Then $N_x \bigcap L$ is a compact torsion-free Lie group, and hence trivial.  Therefore, the induced action of $L$ on $X$ is proper and free. Then the quotient $X/L$ has the structure of a  complex manifold such that the projection $X \rightarrow X/L$ is a holomorphic map.  If the $N$-action on $X$ is also free, we get a holomorphic fiber bundle $X/L \rightarrow X/N$ with fiber $N/L$. Further, if $N$ is abelian then the holomorphic fiber bundle $X/L \rightarrow X/N$ has the structure of a holomorphic principal $(N/L)$-bundle.

  \begin{lemma}\label{exl} Let $N$ be a complex linear algebraic group with $\text{rank}(N) > 1$ and let $K$ be a maximal compact subgroup of $N$.  Then there exists a nontrivial torsion-free closed complex Lie subgroup $L$ of $N$.
 Further, if $N$ is of even rank, then we can choose $L$ to be transverse to $K$. 
  \end{lemma}
  
  \begin{proof} First assume that $N$ is a reductive complex linear algebraic group. Then the lemma follows from Proposition \ref{semisimple} if the maximal compact subgroup $K$  of $N$ is even dimensional. Otherwise, note that
  $K \times S^1$ is an even-dimensional  maximal compact Lie subgroup of the reductive group
  $N \times \mathbb{C}^*$. Then by  
   Proposition \ref{semisimple},  $N \times \mathbb{C}^*$  has a torsion-free closed complex Lie subgroup $L'$ such that 
   $N \times \mathbb{C}^* = (K \times S^1) L'$. This implies $N \bigcap L'$ is a torsion-free closed complex Lie subgroup of $N$.  
  Moreover, by dimension counting, $L:= N \bigcap L' $ is nontrivial.
  
   Now assume that $N$ is a complex linear algebraic group. As noted in the previous section, $N$ is isomorphic to the semi-direct product $G\cdot U$, where $U$ is the unipotent radical of $N$ and $G \cong N/U$ is a reductive Levi subgroup of $N$ containing $K$.  The complex reductive Lie group $G$ contains $K$ as a maximal compact Lie subgroup. 
   If $\pi:N \rightarrow N/U$ is the projection map, then the inverse image $\pi^{-1}(L')$ of a torsion-free closed complex Lie subgroup $L' \in N/U \cong G$ is a torsion-free closed complex Lie subgroup of $N$.  This completes the proof.
  \end{proof}
  
  By the above lemma, any proper holomorphic action of  a complex linear algebraic group $N$ of rank $> 1$ on a complex manifold $X$  yields a complex manifold $X/L$.  Any LVMB manifold satisfying condition (K) (cf. \cite{zaffran}),  and all the examples in Section \ref{mr}  are of the form $X/L$ as above.  Further, in all these examples the corresponding group $N$ acts with at most finite stabilizers so that the quotient $X/N$ has the structure of a complex analytic orbifold.

   Now consider any $n$-dimensional, effective, complex analytic orbifold $V$. We will construct some complex manifolds associated to $V$. 

      Note that  $V$ is the quotient of  a complex manifold $X$ by a proper, holomorphic action of $GL(n,\mathbb{C})$. This follows from the so-called frame bundle construction: Let $\{ V_i : i \in I \}$ be an open cover of $V$. Let $\{(U_i, \Gamma_i, q_i) \}$ be an orbifold atlas for $V$ corresponding to this cover. 
   Here we may assume that
\begin{itemize}
\item $U_i \subset \mathbb{C}^n$,
\item  $\Gamma_i$ is a finite subgroup of complex analytic isomorphisms of  $U_i$
   (In fact,   $\Gamma_i$  can be assumed to be a finite subgroup of $GL(n,\mathbb{C})$ by \cite{Cart}),  
\item   $q_i : U_i \to V_i$ is a continuous map that induces a homeomorphism $\Gamma_i  \backslash U_i \to V_i$.
\end{itemize}

\noindent Let $\widetilde{X}_i$ be the space of frames of the holomorphic tangent bundle of $U_i$. Define $X_i = \Gamma_i \backslash \widetilde{X}_i  $. Then the $X_i$ glue together naturally to give a complex manifold $X$. For details, we refer to \cite{wz} where 
        an analogous construction  is described  in  the presence of a hermitian metric.  There is a natural right action of  $GL(n,\mathbb{C})$  on $X$, which is holomorphic and proper (but not free unless $V$ is a manifold). The quotient $X/GL(n, \mathbb{C})$ is isomorphic 
       to $V$ as a complex orbifold. 
       
         Assume $n$ is even. Then $U(n)$ has even dimension. 
        Then by Proposition \ref{semisimple} there exists a  closed and torsion-free complex Lie subgroup $L$ of $GL(n,\mathbb{C})$ such that 
         $GL(n,\mathbb{C})/L$ is diffeomorphic to $U(n)$.  In this case, the complex manifold $X/ L$ is diffeomorphic to the space of unitary   
         frames,  $\widetilde{V}$, of   \cite{wz}. 
       The fiber of the map $p: X/L \to V$ at a regular point of $V$ is biholomorphic to $GL(n,\mathbb{C})/ L $. The fiber is indeed a complex submanifold of $X/L$. Since $U(n)$ is not K\"ahler, $X/L$ can not be K\"ahler.

       Now, assume $n$ is odd. Then $SU(n)$ is even dimensional. Suppose the complex orbifold $V$ is Calabi-Yau. 
 This means there is a holomorphic volume form $\Omega$  on $V$. So each local chart $U_i$ has a holomorphic volume form
 $\Omega_i$ which is invariant under $\Gamma_i$, and these forms are compatible under gluing of orbifold charts. In particular
each $\Gamma_i$  is a subgroup of $SL(n,\mathbb{C})$. We define $\widetilde{X}_i^{\ast}$ to be the space of those  frames of $T^{\ast (1,0)}U_i$ the wedge product of  whose components equals $\Omega_i$.
Let $X_i^{\ast} =  \Gamma_i \backslash \widetilde{X}_i^{\ast} $. Then the  ${X}_i^{\ast}$ glue together to produce a complex manifold ${X}^{\ast}$ which admits a proper holomorphic action of $SL(n,\mathbb{C})$. 
 By a similar argument as above, there exists a complex subgroup $L$ of $SL(n,\mathbb{C})$ such that $SL(n,\mathbb{C})/L \cong SU(n)$ and
 $X^{\ast}/L$ is a complex non-K\"ahler manifold.  
 Note that, by duality, $\Omega$ induces a holomorphic trivialization  of $\wedge^n T^{(1,0)}V$. Here $T^{(1,0)}V$ denotes the holomorphic
 tangent bundle of $V$ whose sections over $V_i$ are the $\Gamma_i$-invariant sections of  $T^{(1,0)}U_i$. Thus, there exists an invariant, 
 nowhere vanishing, holomorphic poly-vector field $\eta$ of degree $n$ on $V$. Therefore, an analogous construction can be made with tangent frames whose components have wedge product $\eta$.
     
       \begin{remark} If we consider $V$ to be an orbifold toric variety,  then the manifolds $X/L$ and $X^*/L$ are different from the LVMB manifolds associated to $V$. \end{remark}

{\bf Acknowledgement.} The authors are grateful to Indranil Biswas,  Arijit Dey, Amit Hogadi, Shengda Hu, Alberto Medina, M. S. Narasimhan,  Marcel Nicolau, A. Raghuram, Parameswaran Sankaran and Yury Ustinovskiy for helpful discussions, comments and suggestions.  They  thank the 
Indian Institute of Technology-Madras, where the initial part of the work was done, for its hospitality. They also thank an anonymous referee for 
advising them to explore connections with special Hermitian structures. 
The research of the first-named author has been supported by a FAPA grant from the Universidad de los Andes, and an SRP grant from the Middle East Technical University, Northern Cyprus Campus. The research of second-named author has been supported by Indian Statistical Institute, Bangalore and DST-Inspire Faculty Scheme(IFA-13-MA-26).

\end{document}